\documentclass{article}

\usepackage{notes}
\usepackage[unicode=true]{hyperref}
\usepackage{cleveref}
\let\l\lambda
\let\wtilde\widetilde

\def\O{\mathcal O}

\let\b\beta
\def\F{\mathcal F}

\numberwithin{equation}{section}
\let\d\partial

\theoremstyle{theorem}
\renewtheorem{thm}{Theorem}
\renewtheorem{lem}[thm]{Lemma}
\renewtheorem{cor}[thm]{Corollary}
\renewtheorem{prop}[thm]{Proposition}
\renewtheorem{conj}[thm]{Conjecture}
\numberwithin{thm}{section}
\theoremstyle{definition}
\renewtheorem{dfn}[thm]{Definition}
\renewtheorem{exa}[thm]{Example}
\renewtheorem{note}[thm]{Note}
\renewtheorem{nota}[thm]{Notation}
\renewtheorem{quest}[thm]{Question}
\renewtheorem{rem}[thm]{Remark}

\def\L{\mathcal L}

\makeop{td}
\makeop{vol}
\DeclareMathOperator{\SEnd}{\mathcal{E}nd}
\DeclareMathOperator{\SExt}{\mathcal{E}xt}

\title{The tilting property for $F_*^e\O_X$ on Fano surfaces and threefolds}
\author{Devlin Mallory\thanks{This material is based upon work supported by the National Science Foundation under Grant No.~\#1840190.}}
\begin{document}
\maketitle
\abstract{
Let $X$ be a smooth variety over a field of characteristic~$p$.
It is a natural question whether the Frobenius pushforwards $F_*^e\O_X$ of the structure sheaf are tilting bundles. 
We show if $X$ is a  smooth del Pezzo surface of degree $\leq 4$ or a Fano threefold with $\vol(K_X)<24$ over a field of characteristic~$p$, then $\Ext^i(F_*^e\O_X,F^e_*\O_X)\neq 0$ and thus $F_*^e\O_X$ is not  tilting. In the case of a del Pezzo of degree $\leq 4$, we show $F_*^e L$ is never tilting for \emph{any} line bundle $L$.}

\section{Introduction}

Let $X$ be a projective variety. In order to understand the derived category of coherent sheaves $D^b(\Coh(X))$, it is useful to have a \emph{tilting object} $E\in D^b(\Coh(X))$; such an object generates the derived category in an appropriate sense, and moreover has $\Ext^i(E,E)=0$ for $i>0$.
If $E$ is tilting, there is an equivalence
$D^b(\Coh (X))\cong D^b(\Mod (A^\op))$, where $A=\End E$. 
This provides a way to reduce questions about coherent sheaves on $X$ to (hopefully simpler) questions about modules over $A$. Ideally, one could take $E$ to be a vector bundle; such an object is called a \emph{tilting bundle}.
It is thus natural to ask for a natural source of tilting bundles. If $X$ is a regular variety in characteristic $p$, one natural family of bundles is provided by $F_*^e\O_X$, the Frobenius pushforwards of the structure sheaf.
It is thus natural to ask:

\begin{quest}
Let $X$ be a smooth projective variety. When is $F_*^e \O_X$ tilting? 
\end{quest}

This question has been considered largely in representation-theoretic contexts (i.e., for homogeneous varieties); for example, the following statements are known:
\begin{itemize}
\item \cite{quadrictilt} If $X$ is the smooth quadric hypersurface of dimension $n$, then $F_*^e\O_X$ is tilting exactly when (a) $e=1$ and $p>n$; (b) $e=2$, $n=4$ and $p=3$, or (c) $e\geq 2$, $n$ odd, and $p\geq n$.
\item \cite{flagtilt} If $X$ is a partial flag variety of type $(1,n,n+1)$  for $\SL_n$, and $p$ is sufficiently large, then $F_*^e\O_X$ is tilting.
\item \cite{grasstilt} If $X$ is the Grassmannian $\Gr(2,n)$, then $F_*^e \O_X$ is not tilting unless $n=4$ and $p>3$.
\end{itemize}

In this note, we will consider non-homogeneous Fano varieties, and show the following:

\begin{thm}[Theorems~\ref{main1} and \ref{main2}]
Let $X$ be a smooth del Pezzo surface of degree $\leq 4$ or a smooth Fano threefold with $\vol(K_X)< 24$. Then $H^i(X,\SEnd F_*^e\O_X)\neq 0$ for $e\gg 0$ and some $i>0$, and thus $F_*^e \O_X$ is not tilting for $e\gg0$.
\end{thm}

In fact, in the del Pezzo case, we prove the result for all $e>0$, rather than $e\gg0$.
Moreover, in this case we show moreover that $F_*^e L$ is never tilting for \emph{any} line bundle $L$.

This suggests that when $X$ is a Fano variety, $F_*^e\O_X$ often fails to be tilting, particularly in the case where $X$ has nontrivial moduli.

The proof of this theorem is by an analysis of the Euler characteristic of $\SEnd F_*^e\O_X$, via an explicit calculation of the Chern classes of $\SEnd F_*^e\O_X$. 
Following some preliminaries in \Cref{prelim}, these calculations are made in \Cref{cccalc}, and the applications to surfaces and threefolds take place in \Cref{2fold} and \Cref{3fold}.

 These methods may be of independent interest in the study of the properties of $F_*^e\O_X$ and $\SEnd F_*^e\O_X$; in particular, it should not be difficult to generalize the methods here to show that $F_*^e\L$ is not tilting for many ample line bundles $L$. It should also be possible to extend the techniques here to the higher-dimensional case. However, the calculations involved become more complicated, and the Euler characteristic becomes a much coarser tool for detecting higher cohomology as dimension increases. We describe a more conceptual (though incomplete!) approach to this question in \Cref{moduli}, as well as a potential connection to $D$-affinity.

\begin{rem}
After our initial proof of the results above, it came to our attention that
\Cref{main1} contradicts the results of 
Section~5 of \cite{samokhin}, which claims that if $X$ is the blowup of $\P^2$ at $k$ points in general position that $\Ext^i(F_*\O_X,F_*^e\O_X)=0$ for $i>0$. In particular, our results show that this claim fails if $k\geq 6$. In \Cref{comparison} we discuss a gap in the arguments of Section~5 of \cite{samokhin}.
\end{rem}

\section{Preliminaries}
\label{prelim}
In this section we briefly review some of the necessary definitions. The reader is encouraged to skip to \Cref{cccalc} and return here as needed.

\subsection{Differential operators and the Frobenius}

\begin{dfn}
Let $k$ be a field and
$R$ be a $k$-algebra. The $k$-linear differential operators of order $m$, denoted $D^m_{R/k}\subset \End_k(R)$, are defined inductively:
\begin{itemize}
\item $D^0_{R/k} =\Hom_R(R,R)\cong R$, viewed as multiplication by $R$. 
\item  $\delta \in \End_k(R)$ is in $D^m_{R/k}$ if $[\delta, r] \in D^{m-1}_{R/k}$ for any $r\in D^0_{R/k}$.
\end{itemize}
We write $D_{R/k}=\bigcup D^m_{R/k}$. 
\end{dfn}

In characteristic $p$, $D_{R/k}$ has a nice description in terms of the Frobenius.

\begin{prop}[{{\cite[1.4.9]{Yekutieli}}}]
Let $k$ be a perfect field of characteristic $p>0$, and 
and  let $R$ be an $F$-finite $k$-algebra. Then 
$$
D_{R/k} = \bigcup_{e} \Hom_{R^{p^e}}(R,R).
$$
\end{prop}

\begin{rem}
This globalizes to an equality of sheaves
$$
D_{X} = \bigcup_{e} \SHom_{\O_X^{p^e}}(\O_X,\O_X),
$$
where $D_X$ is the sheaf of differential operators on $X$.
We write $D^{(e)}_X$ for $\SHom_{\O_X^{p^e}}(\O_X,\O_X)$, viewed as a left $\O_X$-module by postcomposition.

Moreover, as abelian groups,
$ \SHom_{\O_X^{p^e}}(\O_X,\O_X)$ and $ \SHom_{\O_X}(F_*^e\O_X,F_*^e\O_X) $ are equal; one can check that 
this identification leads to an $\O_X$-module isomorphism
$$
 F_*^eD^{(e)}_X= F_*^e\SHom_{\O_X^{p^e}}(\O_X,\O_X) \cong 
\SHom_{\O_X}(F_*^e\O_X,F_*^e\O_X) .
$$

In particular, $H^i(X,\SEnd_{\O_X}(F_*^e\O_X))=H^i(X,D^{(e)}_X)$ for all $i$ and $e$.
\end{rem}

\begin{rem}
\label{pparts}
We note also that $D_X^{( e )}$ is $\SHom_{\O_X}(\O_{X\times X}/\Delta^{p^{[e]}},\O_X)$, where $\Delta$ is the ideal of the diagonal in $X\times X$ and $\O_{X\times X}/\Delta^{p^{[e]}}$ is viewed as a \emph{left} $\O_X$-module; similarly, 
$D_X^m $ is $\SHom_{\O_X}(\O_{X\times X}/\Delta^{m+1},\O_X)$.
\end{rem}

\subsection{Tilting bundles}

\begin{dfn}
A coherent sheaf $E$ on a variety $X$ is a tilting sheaf if:
\begin{itemize}
\item 
$\Ext^i(E,E)=0$ for all $i>0$,
\item
$\End(E)$ has finite global dimension, and
\item $E$ generates $D^b(\Coh(X))$.
\end{itemize}
$E$ is a tilting bundle if it is tilting and locally free.
\end{dfn}

\begin{rem}
The only condition above that we will use in this note is the first: we will show that $\Ext^i(F_*^e\O_X,F_*^e\O_X)\neq 0$ for many varieties $X$, and thus that $F_*^e\O_X$ is not tilting.
\end{rem}

\begin{rem}
The interest in tilting sheaves (or more generally tilting objects of the derived category) is that if $E$ satisfies the stronger definition mentioned in the preceding remark, and if $A=\End(E)$, then there is an equivalence of derived categories $D^b(\Coh (X))\cong D^b(\Mod (A^\op))$. Thus, tilting sheaves, when they exist, provide a powerful tool for analyzing $D^b(\Coh(X))$.
\end{rem}

The following elementary observation will be useful in computing $\Ext^i(E,E)$:

\begin{lem}
Let $E$ be a locally free sheaf. Then $\Ext^i(E,E)=H^i(\SHom(E,E))$.
\end{lem}

\begin{proof}
Since $E$ is locally free, $\SExt^i(E,E)=0$ for $i>0$. Thus the spectral sequence
$$
H^i(\SExt^j(E,E))\implies \Ext^{i+j}(E,E)
$$
collapses to yield $H^i(\SHom(E,E))=\Ext^i(E,E)$.
\end{proof}

\begin{rem}
Let $X$ be a regular variety over an $F$-finite field of characteristic $p$. Then $F_*^\e\O_X$ is locally free for any $e$ by Kunz's theorem, and thus in particular
$\Ext^i(F_*^e\O_X,F_*^e\O_X)=H^i(\SHom(F_*^e\O_X,F_*^e\O_X))$.
\end{rem}

\subsection{Chern classes and Grothendieck--Riemann--Roch}

The following material is a basic overview of the required tools from the theory of Chern classes and Grothendieck--Riemann--Roch; we include it here because it may be unfamiliar to those primarily concerned with differential operators in commutative algebra.

\begin{dfn}
Let $E$ be a vector bundle on a smooth variety $X$. We write $A(X)$ for the Chow ring of $X$.
The Chern classes of $X$ are elements $c_i(E)\in A^i(X)$.
We write $c_i(X):=c_i(T_X)$.
\end{dfn}

(For a definition and basic properties of the Chow ring, see \cite[Appendix~A]{Hartshorne}.)

\begin{dfn}
The Chern character $\ch(E)$ of a vector bundle $E$ is defined as follows:
For a line bundle $L$, we set $\ch(L)=e^{c_1(L)}=1+c_1(L)+c_1(L)^2/2!+\dots$. It is then extended multiplicatively to arbitrary vector bundles via the splitting principle.
Explicitly,
$$\ch(E)=\rank(E)+c_1(E)+(c_1(E)^2-2c_2(E))/2!+\cdots.$$
\end{dfn}

The key properties of the Chern character we will need are the following:
\begin{itemize}
\item $\ch(E\otimes F)=\ch(E)\cdot \ch(F)$ and $\ch(E\oplus F)=\ch(E)+\ch(F)$.
\item $\ch_i(E^*)=(-1)^i\cdot \ch_i(E)$.
\end{itemize}

\begin{dfn}
The Todd class of a vector bundle $E$, denoted $\td(E)$, is defined as follows: if $L$ is a line bundle, then $\td(L)=\frac{c_1(L)}{1-e^{-c_1(L)}}=1+c_1(L)/2+c_1(L)^2/12+\dots$. If $E=\bigoplus L_i$, then $\td(E)=\prod \td(L_i)$; this suffices to determine the Todd class of any vector bundle via the splitting principle.
Explicitly, 
$$
\td(E)=1+c_1(E)/2
+(c_1(E)^2+c_2(E))/12
+c_1(E)c_2(E)/24+\dots.
$$
Again, we write $\td(X):=\td(T_X)$.
\end{dfn}

\begin{thm}[Grothendieck--Riemann--Roch]
Let $E$ be a coherent sheaf on a smooth projective variety $X$ of dimension $n$. Then
$$
\chi(E):=\sum (-1)^i \dim H^i(E) = \int_X \ch(E)\cdot \td(X) :=\sum_{j} \ch_{n-j}(E)\td_j(X).
$$
\end{thm}

We will make crucial use of the following theorem in calculating Chern classes of Frobenius pushforwards:

\begin{thm}{{\cite[Theorem~6.7]{wonderful}}}
\label{wonderfulthm}
Let $L$ be a line bundle on a smooth projective variety $X$ of characteristic $p$, let $e>0$, and write $\psi_{p^e}$ for the $p^e$-th Adams operation, which acts as multiplication by $p^i$ on an element of $A^i(X)$. Then
$$
\frac{\ch(F_*^e L)}{p^{e n}} 
= \frac{\psi_{p^e}\inv\bigl(\td(X) \cdot \ch(L)\bigr)}{\td(X)},
$$
\end{thm}

\section{Chern class calculations}
\label{cccalc}

Let $X$ be a smooth variety of dimension $n$ in characteristic $p$.
We write $d_i$ for the $i$-th Todd class of $X$, so that $\td(X)=1+d_1+d_2+\cdots$; thus, $d_1=c_1(X)/2$, $d_2=(c_1(X)^2+c_2(X))/12$, $d_3=c_1(X)c_2(X)/24$, and so on.

\begin{prop}
Let $X$ be a smooth variety of dimension $n$ in characteristic $p$.
Then 
$$\ch(X,\SEnd F_*^e\O_X) = 
\let\bigg\Big
p^{2en} d_n + 
(p^{2e(n-1)}-p^{2en})(2d_2-d_1^2 )d_{n-2}+\cdots
.$$
\end{prop}

\begin{rem}
Since the calculations below will be applied in the case of surfaces and threefolds, we will expand out all calculations explicitly to the necessary degrees; the proof will make clear how to obtain the next terms in the formula above.

In order to apply this particular formula, we will need to know the various intersection products of the $d_i$, i.e., of the coefficients of the Todd class. We will not make a general analysis, but rather consider on the special case of Fano surfaces and threefolds.
\end{rem}

\begin{proof}
By Grothendieck--Riemann--Roch,
$$
\chi(\SEnd F_*^e\O_X) = \int \ch\SEnd(F_*^e\O_X)\cdot \td(X) ;
$$

To calculate $\ch\SEnd(F_*^e\O_X)$, we first calculate
$\ch(F_*^e\O_X)$ by
\Cref{wonderfulthm}
$$
\frac{\ch(F_*^e \O_X)}{p^{e n}} 
= \frac{\psi_{p^e}\inv\bigl(\td(X) \cdot \ch(\O_X)\bigr)}{\td(X)}
= \frac{\psi_{p^e}\inv\bigl(\td(X) \bigr)}{\td(X)} .
$$ 
Write
$$
\td(X) = 1+d_1+d_2+d_3+\dots
$$
(only later will it be useful to write $\td(X)$ in terms of $c_i(T_X)$).
Then 
$$
\psi\inv_{p^e}(\td(X)) = 1+p^{-e}d_1+p^{-2e}d_2+p^{-3e}d_3+\dots,
$$
and 
$$
\frac{1}{\td X}= 1-d_1+(d_1^2-d_2) +(-d_3+2d_2d_1-d_1^3)+\dots;
$$
thus 
$$
\frac{\ch(F_*^e \O_X)}{p^{e n}}  = 
(1+p^{-e}d_1+p^{-2e}d_2+p^{-3e}d_3+\cdots)
\cdot (1-d_1+(d_1^2-d_2) +(2d_2d_1-d_1^3)+\cdots).
$$
Expanding this out, we have
$$
\eqalign{
\frac{\ch(F_*^e \O_X)}{p^{e n}}  &= 1+(p^{-e}-1)d_1 + (p^{-2e}-1 )d_2-(p^{-e}-1)d_1^2 +\dots
%\cr&\qquad+ (p^{-e}d_1^3 -p^{-2e}d_1d_2 +p^{-3e}d_3 -d_1^3-p^{-e}d_1d_2 +2d_1d_2)+\dots
}
$$
and thus
$$
\let\bigg\Big
\eqalign{
%\ch(F_*^e \O_X)  &= p^{en}\biggl(1+(p^{-e}-1)d_1 +( p^{-e}-1)((p^{-e}+1 )d_2-d_1^2) \cr&\qquad+ (p^{-e}d_1^3 -p^{-2e}d_1d_2 +p^{-3e}d_3 -d_1^3-p^{-e}d_1d_2 +2d_1d_2)+\dots\biggr)
\ch(F_*^e \O_X)  &= 
\underbrace{p^{en}}_{\ch_0}+\underbrace{(p^{e(n-1)}-p^{en})d_1}_{\ch_1} + \underbrace{(p^{e(n-2)}-p^{en})d_2-(p^{e(n-1)}-p^{en})d_1^2}_{\ch_2}+\dots
%\cr&\qquad+ (p^{-e}d_1^3 -p^{-2e}d_1d_2 +p^{-3e}d_3 -d_1^3-p^{-e}d_1d_2 +2d_1d_2)+\dots\biggr)
}
$$

If $E$ is any vector bundle, then
$$
\eqalign{
\ch(\SEnd E) &= \ch(E\otimes E^*) \cr&= \ch(E)\cdot \ch(E^*)
\cr&= 
(\ch_0(E)+\ch_1(E)+\ch_2(E)+\cdots)
\cdot 
(\ch_0(E^*)+\ch_1(E^*)+\ch_2(E^*)+\cdots)
\cr&= 
(\ch_0(E)+\ch_1(E)+\ch_2(E)+\dots)
\cdot 
(\ch_0(E)-\ch_1(E)+\ch_2(E)+\dots)
\cr&= 
\ch_0(E)^2 + \bigl(2 \ch_2(E) \ch_0(E) - \ch_1(E)^2\bigr)  +\cdots
.
}
$$
In particular, note that all odd-degree terms are 0.

Now, applying this to $E=F_*^e \O_X$, we get
$$
\let\bigg\Big
\medmuskip2.5mu
\eqalign{
\ch(\SEnd F_*^e\O_X) &= 
p^{2en}+
2p^{en}\bigl((p^{e(n-2)}-p^{en})d_2-(p^{e(n-1)}-p^{en})d_1^2\bigr) - (p^{e(n-1)}-p^{en})^2d_1^2 + \cdots
\cr&=
p^{2en}+
2(p^{2e(n-1)}-p^{2en})d_2 - (p^{2e(n-1)}-p^{2en})d_1^2 + \cdots
\cr&=
p^{2en}+
(p^{2e(n-1)}-p^{2en})(2d_2-d_1^2 )
}
$$

Thus, we have that
$$
\chi(\SEnd F_*^e\O_X) =\int \ch( \SEnd F_*^e\O_X)\cdot \td(X) =
\let\bigg\Big
p^{2en} d_n + 
(p^{2e(n-1)}-p^{2en})(2d_2-d_1^2 )d_{n-2}+\cdots
%p^{2en}d_n
%+ 2( p^{2e( n-1)}-1)d_2\cdot d_{n-2}
%-
%\bigl(
%p^{2en}(2p^{-e}+p^{-n})+2p^{en}(1-p^{-e})+1
%\bigr)d_1^2\cdot d_{n-2}
%+\dots
$$
and the claim is proved.
\end{proof}

\section{Del Pezzo surfaces}
\label{2fold}

\begin{thm}
\label{main1}
Let $X$ be a smooth del Pezzo surface of degree $d\leq 4$. Then $F_*^e\O_X$ is not tilting for any $e$.
\end{thm}

\begin{proof}
To obtain nonvanishing of the higher cohomology of $\SEnd F_*^e \O_X$,
it suffices to show that if $X$ is a smooth del Pezzo of degree $\leq d$, then  
$$\chi(F_*^e\O_X)=
\Hom(F_*^e\O_X,F_*^e\O_X)
-\Ext^1(F_*^e\O_X,F_*^e\O_X)
+\Ext^2(F_*^e\O_X,F_*^e\O_X)
$$
 is nonpositive.

When $n=\dim X=2$, we have that the Todd class is $1+c_1(X)/2+(c_1^2(T_X)+c_2(T_X))/12$, i.e., $d_1=c_1(T_X)/2$ and $d_2=(c_1^2(T_X)+c_2(T_X))/12$.
Thus, we have that $d_1^2 = c_1(T_X)^2/4$.
For a degree-$d$ del Pezzo, $c_1(T_X)^2=d$, so $d_1^2 =d/4$, and $d_2=1$. 
In particular, $2d_2-d_1^2=2-d/4$.

Thus, we obtain that 
$$
\chi(F_*^e\O_X) = p^{4e}+(p^{2e}-p^{4e})(2-d/4)
%p^(4*e)* ( 1+(p^(-2*e)-1) -(p^(-2*e)-p^(-e))*d/4  )
=
p^{4e}\cdot \frac{d-4}{4}
+
p^{2e}\cdot \frac{8-d}{4}.
$$
The leading term and next term are
$$
p^{4e}\cdot \frac{d-4}{4}\quad\text{and}\quad
p^{2e}\cdot \frac{8-d}{4}
$$
respectively
Clearly for $e\gg0$ the leading term is negative when $d<4$, so the claim holds for $d<4$ and $e\gg0$. 
In fact when $1<d<4$ we have $(8-d)/4 <7/4< 2$, so 
$$
p^{2e}\frac{8-d}{4} < p^{2e}\cdot 2 = p^{4e}\cdot \frac{2}{p^{2e}}.
$$
If 
$$
p^{4e}\cdot \frac{2}{p^{2e}} \leq 
p^{4e}\cdot \frac{d-4}{4},
$$
then $\chi(\SEnd F_*^e\O_X)\leq 0$.

If $p>2$ or $e>1$, then $2/p^{2e}<1/4<(4-d)/4$, and thus $\chi(\SEnd F_*^e\O_X)<0$. 
Similarly,
if $p=2$ and $e=1$,  and $d<3$, $(4-d)/4\geq 1/2\geq 2/p^{2e}$, so the Euler characteristic is 0; since $H^0(\SEnd F_*^e\O_X)>0$ (since the identity is a global section), again $\SEnd F_*^e\O_X$ must have higher cohomology.

The only missing case when $d=3$ is thus $p=2, e=1$; 
in this case, one can check that $\chi(\SEnd F_*\O_X)=1$. Thus, as long as $h^0(\SEnd F_* \O_X)>1$, $\SEnd F_*\O_X$ must have higher cohomology. If $X$ is $F$-split, then there are multiple indecomposable summands of $\O_X$, so $h^0(\SEnd F_*\O_X)>1$ (since each indecomposable summand of $F_*\O_X$ corresponds to an idempotent global section of $\SEnd F_*\O_X$). Since the non-$F$-split cubic surface $X_0$ can be obtained as the special fiber of a family of $F$-split cubic surfaces $X_t$, upper-semicontinuity says that 
$h^0(\SEnd F_* \O_{X_0})\geq 
h^0(\SEnd F_* \O_{X_t}) >1$. Thus, for any smooth cubic surface $X$, $h^0(\SEnd F_*\O_X)>1$ and thus $\SEnd F_*\O_X$ has higher cohomology.

The last case is $d=4$. This requires a slight variant on the argument above, where we replace the role of $F_*^e\O_X$ by $\coker(F_*^{e_1}\O_X\to F_*^e\O_X) =: \wtilde B_e$. Since $X$ is $F$-split by \cite{HaraFsplit}, $\wtilde B_e$ is a summand of $F_*^e\O_X$, and so it suffices to prove that $\chi(\wtilde B_e) <0$ for any $e>0$.

Since $\ch(\wtilde B_e) + \ch(\O_X) = \ch(F_*^e\O_X)$, we can calculate $\ch(\wtilde B_e)$ and thus $\ch(\End \wtilde B_e)$ by the methods of the preceding section; the result is that
$$
\eqalign{
\ch_0(\SEnd \wtilde B_e ) &= (p^{2e}-p^{2(e-1)})^2
\cr
\ch_1(\SEnd \wtilde B_e ) &= 0
\cr
\ch_2(\SEnd \wtilde B_e ) &= 
-2 (p^{2e}-p^{2(e-1)})^2 d_2
%-(p - 1)^2 p^{2 (e - 1)} (p^{ e} + p^{e-1} - 1) 
+
(p-1)^2 p^{2(e-1)}\bigl( (p^{e}+p^{e-1})^2-1\bigr)
d_1^2
}
$$
Plugging this into Hirzebruch--Riemann--Roch, we get
$$
\chi (\SEnd \wtilde B_e) = 
-(p^{2e}-p^{2(e-1)})^2 + 
(p-1)^2 p^{2(e-1)}\bigl( (p^{e}+p^{e-1})^2-1\bigr) \cdot \frac d 4
$$
Under the assumption $d=4$, this simplifies
to
$$ -p^{2e}+2p^{2e-1}-p^{2e-2}.$$
This is negative for all $p\geq 2$ and $e\geq 1$, and thus $\wtilde B_e$ is not tilting; since it is a direct summand of $F_*^e\O_X$, neither is $F_*^e\O_X$.
\end{proof}

%\begin{rem}[$d=4$]
%The theorem above implies that if $X$ is a del Pezzo of degree 4 in characteristic $p$, then
%$$
%\chi(\SEnd F_*^e \O_X) = p^{2e}>0.
%$$
%Thus, it says nothing about $F_*^e\O_X$ being tilting. 
%However, calculations via Macaulay2 \cite{M2} show that (at least for small $p^e$) that $F_*^e\O_X$ fails to be tilting here as well, and in particular that $H^i(\SEnd F_*^e \O_X)\neq 0$ for $i=1,2$. Thus, the theorem above is not sharp, and different methods must be used to analyze the case of degree-4 del Pezzo surfaces. One such approach is discussed in \Cref{moduli} below.
%\end{rem}

\begin{rem}[$d=5$]
\label{remHara}
The summands in the $d=5$ case are completely analyzed in \cite{Hara1}, where it is shown that $F_*^e\O_X$ is the direct summand (up to multiplicities) of five line bundles, an indecomposable vector bundle $G$ of rank $2$, and an indecomposable vector bundle $B$ of rank 3.
Moreover, $B,G$ are given as extension classes of line bundles in a characteristic-free fashion, and there is always one summand of $B$ and $(p^e-2)(p^e-3)/2$ summands of $G$ in $F_*^e\O_X$.

One can check by direct computation that $\Ext^1(B,G)\neq 0$ (in fact, it is 1-dimensional). Thus, in particular, $F_*^e\O_X$ is not tilting for $p^e> 3$; however, this is not detected numerically by $\chi(F_*^e\O_X)$.
\end{rem}

\begin{rem}
\label{comparison}
The results of this section contradict the claims of Section~5 of \cite{samokhin}, which claims that if $X$ is the blowup of $\P^2$ at any number of points in general position (e.g., a del Pezzo surface) then $F_*^e\O_X$ is tilting.
Here, we point out a gap in the proof in \cite{samokhin}, to clarify the situation: 

The issue is Proposition 5.1, which states that if $x$ is a general point on $X$, which is a blowup of $\P^2$ at some number of points, then the canonical restriction map
$$
H^0\bigl(F^*_e F_*^e (\O_X)\otimes \w_X^{1-p^e}\bigr)\to 
H^0\bigl(F^*_e F_*^e (\O_X)\otimes \w_X^{1-p^e}\otimes \O_{(p^e-1)x}\bigr)
$$
is surjective.
The argument in the proof of Proposition~5.1 only yields that
$$
h^0\bigl(F^*_e F_*^e (\O_X)\otimes \w_X^{1-p^e}\bigr) >
h^0\bigl(F^*_e F_*^e (\O_X)\otimes \w_X^{1-p^e}\otimes \O_{(p^e-1)x}\bigr),
$$
which is not enough to guarantee that the map induced by restriction is itself surjective.

Indeed, we can see that the theorem fails already for $X=\P^n$: in this case, $F_*^e(\O_X)= \O_X\oplus \O_X(-1)^a\oplus\O_X(-2)^b$ for some multiplicities $a,b$, and thus
$$
F^*_e F_*^e (\O_X)\otimes \w_X^{1-p^e} = 
\O_X(3p^{e} - 3) \oplus 
\O_X(2p^{e} - 3)^a \oplus 
\O_X(p^{e} - 3)^b.
$$
Since the canonical restriction map obtained by tensoring with $\O_X\to \O_{(p^e-1)x}$ commutes with direct summands, it  suffices to examine the restriction map on each summand. 

However, it is clear that
$
H^0(\O_X(p^{e} - 3))\to H^0(\O_{(p^e-1)x})
$
cannot be surjective, as 
$$
%\sett{degree-$m$ polynomials} =
 H^0(\O_X(m))\to H^0(\O_{(m+2)x}) = \O_{X,x}/\mathfrak m_{x}^{m+2}
$$
is never surjective.
\end{rem}

\section{An enhancement for del Pezzo surfaces}

In fact, in the surface case we can show something slightly stronger: 

\begin{prop}
Let $X$ be a smooth surface in characteristic $p$.
Then 
$$\ch(\SEnd F_*^eL) = 
%\let\bigg\Big
%p^{4e}+
%p^{2e}(1-p^{2e})(2d_2-d_1^2)
\ch(\SEnd F_*^e\O_X) 
%p^{2en} d_n + 
%(p^{2e(n-1)}-p^{2en})(2d_2-d_1^2 )d_{n-2}+\cdots
$$
for any line bundle $L$.
\end{prop}

From the results of the preceding section, this immediately establishes:

\begin{cor}
If $X$ is a smooth del Pezzo of degree $d\leq 4$, and $L$ \emph{any} line bundle on $X$, then $F_*^e L$ is not tilting.
\end{cor}

\begin{proof}
We imitate the calculations of \Cref{cccalc}:
write $ \ch(L) = 1+l_1+l_2$.
Then
$$
\psi\inv_{p^e}(\td(X)\cdot \ch(L)) = 1+p^{-e}(d_1+l_1)+p^{-2e}(d_2+d_1l_1+l_2)
$$
and 
the expression for
$\frac{1}{\td X}$ is as previously.
Thus
$$
\ch(F_*^e L)  = 
\underbrace{p^{2e}}_{\ch_0}+\underbrace{(p^{e}-p^{2e})d_1+p^e l_1}_{\ch_1} + \underbrace{(1-p^{2e})d_2-(p^{e}-p^{2e})d_1^2+l_2+(1-p^e)l_1d_1}_{\ch_2}+\dots.
$$

Now, applying the formula for $\ch(\SEnd E)$ in terms of $\ch(E)$  to $E=F_*^e L$, we get
$$
\let\bigg\Big
\medmuskip2.5mu
\eqalign{
\ch(\SEnd F_*^eL) &= 
p^{4e}+
2p^{2e}\bigl((1-p^{2e})d_2-(p^{e}-p^{2e})d_1^2+l_2+(1-p^e)l_1d_1\bigr) - ((p^{e}-p^{2e})d_1+p^e l_1)^2.
\cr&=
p^{4e}+
2p^{2e}(1-p^{2e})d_2+
p^{2e}(1-p^{2e})d_1^2+
2p^{2e}l_2-
p^{2e}l_1^2
}
$$
(note that there is no $d_1l_1$ term).

Thus, we have that
$$
\displaylines{
\chi(\SEnd F_*^eL) =\int \ch( \SEnd F_*^eL)\cdot \td(X) 
\hfill\cr\hfill=
\let\bigg\Big
p^{4e} d_2
+
p^{4e}+
2p^{2e}(1-p^{2e})d_2+
p^{2e}(1-p^{2e})d_1^2+
2p^{2e}l_2-
p^{2e}l_1^2
}
%+ 2( p^{2e( n-1)}-1)d_2\cdot d_{n-2}
%-
%\bigl(
%p^{2en}(2p^{-e}+p^{-n})+2p^{en}(1-p^{-e})+1
%\bigr)d_1^2\cdot d_{n-2}
%+\dots
$$
Replacing $l_2$ by $l_1^2/2$, we get
$$
p^{4e} d_2
+
p^{4e}+
2p^{2e}(1-p^{2e})d_2+
p^{2e}(1-p^{2e})d_1^2.
$$
Collecting terms, this becomes
$$
\chi(\SEnd F_*^e L)=
p^{4e}+
p^{2e}(1-p^{2e})(2d_2-d_1^2);
%+
%2p^{2e}l_1^2
$$
note there is no dependence on the Chern classes $l_i$ of $L$ whatsoever, and thus the resulting
number is unchanged by replacing $L$ by $\O_X$, establishing the proposition.
\end{proof}

\section{Fano threefolds}
\label{3fold}

\begin{thm}
\label{main2}
Let $X$ be a smooth Fano threefold. If $\vol(-K_X)< 24$, then $F_*^e\O_X$ is not tilting for $e\gg0$.
\end{thm}

%\begin{rem}
%The condition that $\chi(\O_X)=1$ is guaranteed if $X$ is globally $F$-split, or if $X$ lifts to characteristic 0 (since either condition implies that Kodaira vanishing holds, which gives $\chi(\O_X)=1$). In fact, by deformation-invariance of $\chi(\O_X)$ in flat families, it would suffice for $X$ to belong to a family for which the general member satisfies Kodaira vanishing. To my knowledge, there is no known example of a Fano threefold for which this is not true.
%
%\end{rem}

\begin{rem}
The condition that $\vol(-K_X)< 24$ is true for many Fano threefolds; in particular, for prime Fano threefolds (those of Picard number 1), it holds for 12 of the 17 families occurring in characteristic 0 (all except for del Pezzo threefolds of degree $\geq 3$, the quadric, and $\P^3$ itself). Thus, the theorem provides a large class of Fano varieties for which $F_*^e\O_X$ is not tilting.
\end{rem}

\begin{proof}
Again, it suffices to show that the condition $\vol(-K_X)< 24$ implies that $\chi(\SEnd F_*^e\O_X)\leq 0$ for $e\gg0$.

When $n=\dim X=3$, we have that the Todd class is $1+c_1(T_X)/2+(c_1^2(T_X)+c_2(T_X))/12+c_1(T_X)c_2(T_X)/24$, i.e., $d_1=c_1(X)/2$, $d_2=(c_1^2(T_X)+c_2(T_X))/12$, and $d_3=c_1(T_X)c_2(T_X)/24$.
We need to understand the values of the intersection products $d_3$, $d_2\cdot d_1$, and $d_1^3$. It is clear that $d_1^3 = (-K_X)^3/8$.
Moreover,
Corollary 1.5 of [SB] guarantees smooth Fano threefolds have $\chi(\O_X)=1$, and thus
 $c_1(T_X)c_2(T_X)=24$ (by Riemann--Roch for $\O_X$).
Thus, we have that $d_3=1$ and $d_2d_1=(-K_X)^3/24+1$.
In particular,
$$(2d_2-d_1^2 )\cdot d_1 = \frac{(-K_X)^3}{12}+2-\frac{(-K_X)^3}{8}
=\frac{48-(-K_X)^3}{24}
=\frac{48-\vol(-K_X)}{24}
 $$
Thus, we have that
$$
\chi(\SEnd F_*^e\O_X) 
%=\int \ch( \SEnd F_*^e\O_X)\cdot \td(X) =
\let\bigg\Big
=
p^{6e}+ 
(p^{4e}-p^{6e})\cdot 
\frac{48-\vol(-K_X)}{24}
=
p^{6e}\cdot \frac{\vol(-K_X)-24}{24}
+
p^{4e}\cdot \frac{48-\vol(-K_X)}{24}.
$$
The leading term is negative exactly for $\vol(-K_X)< 24$, giving the theorem.
\end{proof}

\section{Varieties with moduli, $D$-affinity, and other open questions}
\label{moduli}

Motivated by these examples, we make the following conjecture:

\begin{conj}
\label{conj1}
Let $X$ be a smooth projective variety with nontrivial first-order deformations (i.e., with $h^1(T_X)>0$). Then $F_*^e\O_X$ is not tilting for any $e>0$.
\end{conj}

\begin{rem}

We describe a possible method of proof for this conjecture.
First, note that for any smooth variety $X$, we have a short exact sequence
$$
0\to \Delta^2/\Delta^{[p^e]}\to \O_{X\times X}/\Delta^{[p^e]}\to \O_{X\times X}/\Delta^2\to0,
$$
where $\Delta$ is the ideal of the diagonal in $X\times X$.
Viewing these as $\O_X$-modules via restriction of scalars along the left inclusion $\O_X\to \O_{X\times X}$, 
we obtain a short exact sequence of locally free sheaves on $X$.
Taking $\O_X$-duals, by \Cref{pparts} we obtain 
$$
0\to D^1_X \to D^{(e)}_X \to \SHom(\Delta^2/\Delta^{[p^e]},\O_X) \to 0,
$$
where $D^{(e)}_X$ is the sheaf of $O_X^{p^e}$-linear endomorphisms of $\O_X$.

Note that the canonical short exact sequence $0\to \O_X \to D^1_X \to T_X\to 0$ splits via the map $D^1_X\to \O_X$ given by evaluation at 1. Thus, we have $D^1_X= \O_X\oplus T_X$, and in particular 
$H^1(D^1_X) =H^1(\O_X)\oplus H^1(T_X)$. By assumption, $H^1(T_X)\neq 0$ and thus $H^1(D^1_X)\neq 0$.

Now, the hope would be to show that the map $\gamma:H^1(D^1_X)\to H^1(D^{(e)}_X)$ is not the zero map. The point is that $F_*^e D^{(e)}_X = \SEnd F_*^e\O_X$, and thus if $D^{(e)}_X$ has nontrivial cohomology then $\SEnd F_*^e\O_X$ is not tilting.

The map $\gamma$ fits into the sequence
$$
\displaylines{
0\to H^0(D^1_X) \to H^0(D^{(e)}_X) \xra \a H^0(\SHom(\Delta^2/\Delta^{[p^e]},\O_X))
\hfill\cr\hfill
\xra \beta H^1(D^1_X) \xra \gamma H^1(D^{(e)}_X) \to H^1(\SHom(\Delta^2/\Delta^{[p^e]},\O_X))
}
$$
In particular,
we have $\ker \g = \im \b$, so it would suffice to show $\b$ is not surjective.

Unfortunately, without a better understanding of the sheaves $\Delta^2/\Delta^{[p^e]}$ (and their duals) it seems hard to control the surjectivity of $\b$ or the injectivity of $\g$, and so we are not able to conclude that $H^1(D^{(e)}_{X})\neq 0$. 
\end{rem}

Finally, we conclude with a potential connection to $D$-affinity. Recall that a variety $X$ is $D$-affine if every $D_X$-module $M$ is generated by global sections and $H^i(M)=0$ for $i>0$. In particular, we must have that $H^i(D_X)=0$.
$D$-affinity is known to be a restrictive condition; for example, by \cite{Langer3} if $\Char k >7$, then the only smooth projective $D$-affine surfaces are $\P^2$ and $\P^1\times\P^1$.

The following conjecture would allow us to connect $D$-affinity and the tilting property of $F_*^e\O_X$:

\begin{conj}
Let $X$ be a globally $F$-split variety.
\begin{enumerate}
\item 
for $e'>e$ we have that $H^i(D^{(e)}_X)$ is a direct summand of $H^i(D^{(e')}_X)$, and thus:
\item
 if $F_*^e \O_X$ is not tilting for some $e$, then $X$ is not $D$-affine.
\end{enumerate}
\end{conj}

Note that the second part of the conjecture follows from the first, as $H^i(D_X)=\lim H^i(D^{(e')}_X)$.

If the conjecture is true, then as a consequence the varieties dealt with in Theorems~\ref{main1} or~\ref{main2} are not $D$-affine.

\begin{rem}
The claim of the conjecture appears in Proposition~1.2 of \cite{DaffineB2}, but we believe there to be a gap in the proof there. The issue is as follows: In \cite[1.1]{DaffineB2}, the authors point out that if $X$ is globally $F$-split, then $D^{(e)}_X$ is a direct summand of $D^{(e')}_X$ as sheaves of abelian groups; this follows by viewing $D^{(e)}_X$ as $\SEnd F_*^e \O_X$, and likewise for $D^{(e')}_X$, and noting that since $X$ is $F$-split, $F_*^e \O_X$ is an $\O_X$-linear direct summand of $F_*^{e'}\O_X$, and thus the same is true for their endomorphism sheaves.

However, the natural inclusion 
$i:D^{(e)}_X\hookrightarrow D^{(e')}_X$ is \emph{not} the same as the inclusion of $\SEnd F_*^e\O_X$ in $\SEnd F_*^{e'}\O_X$ (and in fact the latter requires the fact that $X$ is globally $F$-split to define).
We provide an example here for clarity:

\begin{exa}
\def\F{\mathbb F}
Let $R=\F_2$ and
Let $X=\Spec R$. %, and fix a splitting $ F_*\O_X\to \O_X$, $\rho F_*(1)=1$ and $\rho F_*(t)=0$.  
Let $i:D^{(1)}_X\to D^{(2)}_X$ denote the natural inclusion, and $j:D^{(1)}_X \to D^{(2)}_X$ for the map obtained via the inclusion $\End F_*R \hookrightarrow \End F_*^2R$ and the identifications $D^{(e)}_R\cong F_*^e R$.
We claim $i$ and $j$ are not identical. 

Note that in this case we have 
$D^{(1)}_X=R\<1,\d/\d t\>$ and 
$D^{(2)}_X=R\<1,\d/\d t, \frac12 (\d/\d t)^2,\frac12 (\d/\d t)^3\>$.
%F_* 1 -> F^2_* 1, F_*x -> F_*^2 x^2
The image of the split inclusion $F_*F: F_*R\to F_*^2R$ is the submodule spanned by $F_*^2 1,F_*^2 t^2$, and the splitting $\tau: F_*^2 R\to F_*R $ send $F^2_*t$ and $F^2_*t^3$ to zero.
The resulting inclusion  $\End F_* R\to \End F^2_* R$ is obtained by sending $\psi \in \End F_* R$ to $F_*F \circ \psi \circ \tau$.

Thus, for example, $\d/\d t$, viewed as the element of $\End F_*R$ sending $1\mapsto 0$ and $t\mapsto 1$, is sent to the element of $\End F_*^2R$ sending $1\mapsto 0, t\mapsto 0, t^2\mapsto 1$ and $t^3\mapsto 0$.  One can check that this is the element $\frac{1}{2} (\d/\d t)^2 +t \frac{1}{2} (\d/\d t)^3$ of $D^{(2)}_X$. In other words, $j(\d/\d t)=\frac{1}{2} (\d/\d t)^2 +t \frac{1}{2} (\d/\d t)^3\neq i(\d/ \d t)$, and thus $j\neq i$.
\end{exa}

Thus, the fact that there is a split inclusion $\SEnd F_*^e\O_X\to\SEnd F_*^{e'}\O_X$ does not imply that 
$i:D^{(e)}_X\hookrightarrow D^{(e')}_X$ splits, and thus does not imply that the maps $H^i(D^{(e)}_X) \to  H^i(D^{(e')}_X)$ are injective.
\end{rem}

It seems that it remains open then whether the natural inclusions $D^{(e)}\to D^{(e')}$ are split for $e'>e$. One could also hope to at least show that 
$$
H^i(D^{(e)}_X)\to H^i(D^{(e')}_X)
$$
is injective for $e'\gg e$, even if $D^{(e)}_X\to D^{(e')}_X$ is not split.
In any case, if one could obtain either of these results, then one could conclude immediately that whenever $F_*^e\O_X$ is not tilting for $e\gg0 $ that $X$ is not $D$-affine.

\begin{rem}
Let $X_0$ be a smooth variety defined over a field of characteristic 0, and $X_p$ its reduction modulo $p$, for some prime $p$ that is smooth. If $F_*^e\O_{X_p}$ is tilting, then in particular $\Ext^i(F_*^e\O_{X_p},F_*^e\O_{X_p})=0$ for $i\geq 1$; as a consequence, the deformations of $F_*^e\O_{X_p}$ are both unobstructed (since $\Ext^2(F_*^e\O_X,F^e_*\O_X)=0$) and unique  (since $\Ext^1(F_*^e\O_X,F^e_*\O_X)=0$). Thus, $F_*^e\O_{X_p}$ lifts to a vector bundle on $X_0$.
Put another way, if $X$ is a smooth variety in characteristic $p$ that lifts to characteristic 0, then if $F_*^e\O_X$ is tilting then $F_*^e\O_X$ can be lifted to characteristic 0 as well.

The tilting condition is sufficient for $F_*^e\O_X$ to lift to characteristic 0, but not necessary; 
for example, on any toric variety, $F_*^e\O_X$ will lift to characteristic 0, whether or not it is tilting (for example, if $X$ is the toric variety $\P(\O_{\P^2}\oplus \O_{\P^2}(2))$, then $\Ext^1(F_*^e\O_X,F_*^e\O_X)\neq 0$  for $p^e\geq 3$). The same is true for the del Pezzo of degree 5 by the explicit description of the summands of $F_*^e\O_X$ by \cite{Hara1} and the remark on its tilting property in \Cref{remHara}.
\end{rem}

This prompts the following question:

\begin{quest}
For a del Pezzo surface of degree $\leq 4$ in characteristic $p$, does $F_*^e\O_X$ lift to characteristic 0?
\end{quest}

Note that if $X$ is $F$-split, this is the same as asking that $\coker(\O_X\to F_*^e\O_X)$ lifts to characteristic 0, or equivalently that the kernel of the Frobenius trace $F_*\w_X\to \w_X$ lifts to characteristic 0.

\bibliographystyle{alpha}
\bibliography{link}

\begin{thebibliography}{R{\v{S}}VdB19}

\bibitem[AK00]{DaffineB2}
Henning~Haahr Andersen and Masaharu Kaneda.
\newblock On the {$D$}-affinity of the flag variety in type {$B_2$}.
\newblock {\em Manuscripta Math.}, 103(3):393--399, 2000.

\bibitem[CK22]{wonderful}
Merrick Cai and Vasily Krylov.
\newblock Decomposition of frobenius pushforwards of line bundles on wonderful
  compactifications, 2022.

\bibitem[Har77]{Hartshorne}
Robin Hartshorne.
\newblock {\em Algebraic {Geometry}}, volume~52 of {\em Graduate {Texts} in
  {Mathematics}}.
\newblock Springer, New York, NY, 1977.

\bibitem[Har98]{HaraFsplit}
Nobuo Hara.
\newblock A characterization of rational singularities in terms of injectivity
  of {Frobenius} maps.
\newblock {\em Am. J. Math.}, 120(5):981--996, 1998.

\bibitem[Har15]{Hara1}
Nobuo Hara.
\newblock Looking out for frobenius summands on a blown-up surface of {P2}.
\newblock {\em Illinois Journal of Mathematics}, 59:115--142, January 2015.

\bibitem[Lan08]{quadrictilt}
Adrian Langer.
\newblock {$D$}-affinity and {F}robenius morphism on quadrics.
\newblock {\em Int. Math. Res. Not. IMRN}, (1):Art. ID rnm 145, 26, 2008.

\bibitem[Lan21]{Langer3}
Adrian Langer.
\newblock On smooth projective {D}-affine varieties.
\newblock {\em Int. Math. Res. Not. IMRN}, (15):11889--11922, 2021.

\bibitem[R{\v{S}}VdB19]{grasstilt}
Theo Raedschelders, {\v{S}}pela {\v{S}}penko, and Michel Van~den Bergh.
\newblock The {F}robenius morphism in invariant theory.
\newblock {\em Adv. Math.}, 348:183--254, 2019.

\bibitem[Sam10a]{samokhin}
Alexander Samokhin.
\newblock Tilting bundles via the frobenius morphism, 2010.

\bibitem[Sam10b]{flagtilt}
Alexander Samokhin.
\newblock A vanishing theorem for differential operators in positive
  characteristic.
\newblock {\em Transform. Groups}, 15(1):227--242, 2010.

\bibitem[Yek92]{Yekutieli}
Amnon Yekutieli.
\newblock An explicit construction of the {G}rothendieck residue complex.
\newblock {\em Ast\'{e}risque}, (208):127, 1992.
\newblock With an appendix by Pramathanath Sastry.

\end{thebibliography}
\end{document}